\pdfoutput=1                            % to force arxiv to use pdflatex
\def\lcversion{}					    % version number

\documentclass[nonote,noversion,squeezed]{cdcarticle}

\def\MODE{2}

\usepackage[utf8]{inputenc}				% handle accents properly
\usepackage[english]{babel}				% handle hyphenation properly

\usepackage{math}						% custom math macros
\usepackage{enumitem}
\usepackage{cite}
\usepackage[font=small,labelfont=bf]{caption}

\usepackage[hidelinks]{hyperref}
\usepackage{subcaption}

\usepackage{mdframed}
\newmdtheoremenv[innertopmargin=0,innerbottommargin=5,%
innerleftmargin=5,innerrightmargin=5]{fthm}[thm]{Theorem}

\usepackage{tikz,pgfplots}
\usetikzlibrary{external,calc,positioning}
\usepgfplotslibrary{fillbetween}
\usepgfplotslibrary{external}
\pgfplotsset{compat=1.14}
\graphicspath{{./graphics/}}

\let\oldbibliography\thebibliography
\renewcommand{\thebibliography}[1]{\oldbibliography{#1}
\setlength{\itemsep}{1pt}} %Reducing spacing in the bibliography.

\newcommand{\ip}[2]{\left\langle #1 ,\, #2 \right\rangle}

\usepackage{booktabs}
\newcommand{\bbmat}[1]{{\renewcommand{\arraystretch}{1.0}\addtolength{\arraycolsep}{-1mm}\bmat{#1}}}
\newcommand{\leftbox}[1]{\parbox{4.6cm}{#1}}

\newcommand{\matbox}[1]{\begin{minipage}[m]{4.0cm}\centering$\bbmat{#1}$\\[-0.5\baselineskip]\end{minipage}}
\newcommand{\midbox}[1]{\begin{minipage}[m]{3.4cm}\centering#1\end{minipage}}
\newcommand{\midboxx}[1]{\begin{minipage}[m]{3.0cm}\centering#1\end{minipage}}

\newcommand\blfootnote[1]{%
	\begingroup
	\renewcommand\thefootnote{}\footnote{#1}%
	\addtocounter{footnote}{-1}%
	\endgroup
}

\begin{document}

\title{Unified Necessary and Sufficient Conditions for the Robust Stability of
	Interconnected Sector-Bounded Systems}

\if\MODE1
\author{Saman Cyrus, Laurent Lessard% <-this % stops a space
	\thanks{This work was supported by National Science Foundation (CISE/CCF). Award
		Number ??.}% <-this % stops a space
	\thanks{The authors are with the Department of Electrical and Computer
		Engineering,
		University of Wisconsin--Madison, Madison, WI 53706, USA.
	}%
	\thanks{The authors are with the Optimization Group at the Wisconsin
		Institute for Discovery.
		Email: {\tt\small \{cyrus2,laurent.lessard\}@wisc.edu}}%
}

\markboth{IEEE Transactions on Automatic Control}%
{}
\else
\author{Saman~Cyrus \and Laurent~Lessard}
\note{}
\fi

\maketitle

%%%%%%%%%%%%%%%%%%%%%%%%%%%%%%%%%%%%%%%%%%%%%%%%%%%%%%%%%%%%%%%%%%%%%%%%%%%%%%%%
\begin{abstract}
Classical conditions for ensuring the robust stability of a linear system in feedback with a sector-bounded nonlinearity include small gain, circle, passivity, and conicity theorems. In this work, we present a similar stability condition, but expressed in terms of relations defined on a general semi-inner product space. This increased generality leads to a clean result that can be specialized in a variety of ways. First, we show how to recover both sufficient and necessary-and-sufficient versions of the aforementioned classical results. Second, we show that suitably choosing the semi-inner product space leads to a new necessary and sufficient condition for weighted stability, which is in turn sufficient for exponential stability.
\end{abstract}

\blfootnote{S.~Cyrus and L.~Lessard are with the Department of Electrical and Computer Engineering and with the Wisconsin Institute for Discovery, both at the University of Wisconsin--Madison, Madison, WI 53706. Email addresses:\\
{\small\texttt{\{cyrus2,laurent.lessard\}@wisc.edu}}\\[1mm]
	This material is based upon work supported by the National
	Science Foundation under Grant No. 1710892.}

%%%%%%%%%%%%%%%%%%%%%%%%%%%%%%%%%%%%%%%%%%%%%%%%%%%%%%%%%%%%%%%%%%%%%%%%%%%%%%%%

\section{Introduction}
In this paper, we consider the feedback interconnection shown in Figure \ref{fig:StableNegativefeedback}, where a linear system $G$ is in feedback with a sector-bounded nonlinearity $\Phi$, and we are interested in conditions that guarantee closed-loop stability. 

Different forms of this problem have been a point of study for over 75 years since the early work of Lur'e \cite{lur1944theory}. Results typically take the form of fixing conditions on one of the system and describing conditions on the other system that guarantees stability of the closed loop.
Depending on the orientation of the conic sector characterizing $\Phi$, we can obtain passivity~\cite{desoer2009feedback}, small-gain~\cite{desoer2009feedback,Zames_inputoutput1966_part1}, circle~\cite{khalil2002nonlinear}, conic sector~\cite{Zames_inputoutput1966_part1}, and extended conic sector~\cite{bridgeman2016extended} theorems. These results are \textit{sufficient} conditions for stability.
When $G$ is linear and time-invariant (LTI) and $\Phi$ is sector-bounded and memoryless, we obtain the classical Lur'e formulation, where the circle criterion~\cite{brockett1966status} and passivity theorem \cite[Thm.~5.6.18]{vidyasagar2002nonlinear} are sufficient but not necessary for stability. However, if $\Phi$ is allowed to have dynamics, the circle criterion \cite[Thm.~6.6.126]{vidyasagar2002nonlinear} and passivity theorem \cite{khong2018converse} become both necessary and sufficient.

\tikzstyle{block} = [draw, fill=white, rectangle, minimum height=2.8em,
minimum width=2.8em]
\tikzstyle{sum} = [draw, fill=white, circle]
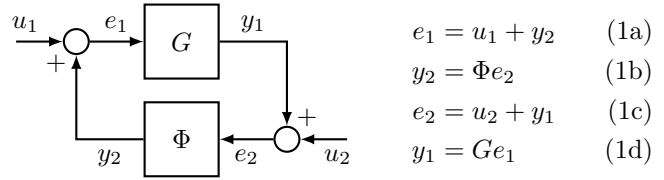
\begin{figure}[!htb]
\centering
\begin{minipage}[T]{0.55\linewidth}
	\centering
	\begin{tikzpicture}[auto, >=latex, thick]
	\node [block] (G) {$G$};
	\node [block] (Phi) at ($ (G) + (0,-1.3) $) {$\Phi$};
	\node [sum] (sum1) at ($ (G) + (-1.4,0) $) {};
	\node [sum] (sum2) at ($ (Phi) + (1.4,0) $) {};
	\node [coordinate] (input) at ($ (sum1) + (-.8,0) $) {};
	\node [coordinate] (output) at ($ (sum2) + (.8,0) $) {};
	\draw [->] (input) -- node[pos=0.2]{$u_1$} (sum1);
	\draw [->] (sum1) -- node{$e_1$} (G);
	\draw [->] (G) -| node[pos=0.25]{$y_1$} node[pos=0.95]{$+$} (sum2);
	\draw [->] (output) -- node[pos=0.2]{$u_2$} (sum2);
	\draw [->] (sum2) -- node{$e_2$} (Phi);
	\draw [->] (Phi) -| node[pos=0.25]{$y_2$} node[pos=0.95]{$+$} (sum1);
	\end{tikzpicture}
\end{minipage}%
\begin{minipage}[T]{0.45\linewidth}
\vspace{-1em}
\begin{subequations}\label{interconnect}
	\begin{align}
		e_1 &= u_1 + y_2 \label{1a}\\
		y_2 &= \Phi e_2 \label{1b}\\
		e_2 &= u_2 + y_1 \label{1c}\\ 
		y_1 &= G e_1 \label{1d}
	\end{align}
\end{subequations}
\end{minipage}
\caption{Two interconnected systems. We assume in this work that $G$ is linear and $\Phi$ is a sector-bounded nonlinearity.\label{fig:StableNegativefeedback}}
\vspace{-2mm}
\end{figure}

Sector bounds are usually defined as bounds on cumulative sums of inner products on $\Ltwoe$, but they can also be defined as holding pointwise in time~(see for example \cite[\S6.1]{khalil2002nonlinear} and \cite[\S1]{desoer2009feedback}). Results can also be formulated in discrete or continuous time.

Finally, one can use a different notion of stability. For example, recent works have developed conditions that guarantee robust \textit{exponential} stability. Under mild assumptions, input-output stability automatically implies exponential stability~\cite{Rantzer97systemanalysis,jonsson1996systems}. However, constructing an exponential rate via a gain bound is conservative in general~\cite{boczar2017exponential}. Less conservative sufficient conditions appeared in~\cite{boczar2017exponential,BinITAC2016} but it is not known whether these conditions are also necessary. These issues arose in the context of analyzing iterative algorithms~\cite{lessard2016analysis,Cyrus2018}, where it is desirable to have tight bounds on worst-case convergence rates.

\paragraph{Main contribution.} Our main contribution is a robust stability theorem that unifies and generalizes many of the aforementioned classical results by distilling them down to their fundamental components (Theorem~\ref{thm:main}). We work in a general \textit{semi-inner product space} (See Section~\ref{sec:mainpre}), we define $G$ and $\Phi$ as \textit{relations} rather than operators, and our result is both necessary and sufficient. The added generality we provide leads to a clean result that avoids the usual technicalities associated with the extended spaces $\Ltwoe$ and $\ltwoe$. Indeed, our setting need not include a notion of time, so there is no need to worry about causality, boundedness, or even well-posedness.

In Section~\ref{sec:Specialization}, we show how classical results in $\ltwoe$, including sufficient-only cases, follow directly from Theorem~\ref{thm:main}. We also clarify exactly how and when causality, boundedness, and well-posedness come into play. In Section~\ref{sec:Exponential}, we use Theorem~\ref{thm:main} to obtain a new weighted stability result that is both necessary and sufficient.

\paragraph{Related work.} Several prior works have also provided unified versions of classical robust stability results.
We cite two examples: the extended conic sector theorem, which can handle the case where $G$ is unstable~\cite{bridgeman2016extended}, and a loop-shifting transformation that relates passivity, small-gain, and circle theorems~\cite{anderson1972small}.
Nevertheless, the present work is unique in its use of semi-inner product spaces and its ability to address exponential stability.

There are also generalizations to cases where $G$ is nonlinear or $\Phi$ is not sector-bounded. Examples include dissipativity theory~\cite{willems1972dissipative}, integral quadratic constraints~\cite{megretski1997system,pfifer2015integral}, and graph separation theorems~\cite{teel1996input,safonov1980stability}. These efforts lie beyond the scope of the present work.

Finding necessary and sufficient stability guarantees has been a point of interest in different applications, including robotics \cite{stramigioli2015energy,colgate1988robust}, robust control \cite[p.~212]{zhou1996robust}, \cite[p.~158]{narendra1973frequency},  and in finding tight upper bounds for the convergence rate of iterative optimization algorithms \cite[\S7]{lessard2016analysis}.

\section{Main result}\label{sec:mainpre}
\paragraph{Semi-inner products.}
	A \textit{semi-inner product space} is a vector space $\mathcal{X}$ equipped with a semi-inner product $\ip{\cdot}{\cdot}$. This is identical to an inner product except that it lacks definiteness. In other words, the associated semi-norm $\norm{x}^2\defeq\langle x,x \rangle$ satisfies $\norm{x}\ge 0$ but $\|x\|=0$ need not imply that $x=0$. We say the semi-inner product space is \textit{\textbf{nontrivial}} if the set $\set{x\in\mathcal{X}}{\norm{x}>0}$ is nonempty.	We refer the reader to~\cite{conway} for further details.
	
	\paragraph{Relations.} A \textit{relation} $R$ on $\mathcal{X}$ is a subset of the product space $R \subseteq \mathcal{X}\times \mathcal{X}$. We denote the set of all relations on $\mathcal{X}$ as $\mathcal{R}(\mathcal{X}) \defeq 2^{\mathcal{X}\times\mathcal{X}}$. We write $Rx$ to denote any $y\in\mathcal{X}$ such that $(x,y)\in R$. A relation $L\in\mathcal{R}(\mathcal{X})$ is \textit{\textbf{linear}} if it has the property that $(\lambda_1x_1+\lambda_2 x_2, \lambda_1 y_1+\lambda_2 y_2)\in L$ for all $(x_1,y_1)$, $(x_2,y_2)\in L$ and $\lambda_1$, $\lambda_2 \in \mathbb{R}$. We define $\mathcal{L}$ to be the set of all linear relations, so $L\in\mathcal{L}\subseteq \mathcal{R}(\mathcal{X})$.
	
	We define $\mathcal{X}^2$ to be the augmented vectors $u\eqdef\left(\begin{smallmatrix}u_1\\u_2\end{smallmatrix}\right)$ where $u_1,u_2\in\mathcal{X}$. We overload matrix multiplication to have an intuitive interpretation in~$\mathcal{X}^2$. Specifically, for any $\xi, \zeta\in\mathcal{X}^2$ and any matrix $N\in\R^{2\times 2}$,
	\[
	N\xi = \bmat{N_{11} & N_{12} \\ N_{21} & N_{22} }\bmat{\xi_1 \\ \xi_2} \defeq \bmat{ N_{11} \xi_1 + N_{12} \xi_2 \\ N_{21} \xi_1 + N_{22} \xi_2} \in \mathcal{X}^2.
	\]
	Likewise, inner products in $\mathcal{X}^2$ have the interpretation
	\[
	\ip{\xi}{\zeta} = \ip{\bmat{\xi_1\\ \xi_2}}{\bmat{\zeta_1\\\zeta_2}} \defeq \ip{\xi_1}{\zeta_1} + \ip{\xi_2}{\zeta_2}.
	\]
	The closed-loop system of Figure \ref{fig:StableNegativefeedback} defines relations:
	\begin{subequations}\label{eq:relations}
		\begin{align*}
		R_{uy} &\defeq \left\{(u,y)\in \mathcal{X}^2\times\mathcal{X}^2 \,|\,\eqref{interconnect} \text{ holds for some }e\in\mathcal{X}^2 \right\}\\
		R_{ue} &\defeq \left\{(u,e)\in \mathcal{X}^2\times\mathcal{X}^2\,|\,\eqref{interconnect} \text{ holds for some }y\in\mathcal{X}^2 \right\}
		\end{align*}
	\end{subequations}
	We call a set of relations $\mathcal{C}\subseteq \mathcal{R}(\mathcal{X})$ \textit{\textbf{feedback-invariant}} if $\set{(u_i,y_j)}{(u,y)\in R_{uy}} \in \mathcal{C}$ for all $G,\Phi\in \mathcal{C}$ and for all $i,j \in \{1,2\}$.
	We call $\mathcal{C}$ \textit{\textbf{complete}} if given any $x,y \in \mathcal{X}$, there exists $\Phi\in\mathcal{C}$ such that $(x,y)\in\Phi$.
\newpage
	
\begin{fthm}[Main result]\label{thm:main}
	Let $\mathcal{X}$ be a nontrivial semi-inner product space, let	$M =M^\tp \in \R^{2\times 2}$ be given and let $\mathcal{C}\subseteq \mathcal{R}(\mathcal{X})$ be complete and feedback-invariant. Suppose $G \in \mathcal{L}\cap\mathcal{C}$. The following are equivalent.
	\begin{enumerate}[label=(\roman*),itemindent=0pt,labelindent=0pt]
		\item\label{thm_it_i} There exists $N=N^\tp\! \in \R^{2\times 2}$ satisfying $M\!+\!N\!\prec \!0$ (positive definite sense) such that $G$ satisfies
		\begin{equation}\label{G}
		\ip{ \bmat{G\xi\\ \xi} }{ N \bmat{G\xi\\ \xi} } \ge 0 \qquad\text{for all
		}\xi\in\mathcal{X}.
		\end{equation}
		\item\label{thm_it_ii} There exists $\gamma>0$ such that for all $\Phi\in\mathcal{C}$, if
		\begin{equation}\label{Phi}
		\ip{ \bmat{\xi\\\Phi \xi} }{ M \bmat{\xi\\\Phi \xi} } \ge 0 \qquad\text{for all
		}\xi\in\mathcal{X},
		\end{equation}
		then for all $(u,y)\in R_{uy}$, the following bound holds
		\begin{equation}\label{norm}
		\norm{y} \le \gamma \norm{u}.
		\end{equation}
	\end{enumerate}
\end{fthm}

\begin{proof}
See Appendix~\ref{sec:appendix1} for a detailed proof.
\end{proof}

\begin{rem}
Equation~\eqref{norm} can be stated in terms of $(u,e)$ instead of $(u,y)$. Specifically,
it is easy to show that \eqref{norm} holds for all $(u,y)\in R_{uy}$ if and only if there exists some $\bar\gamma > 0$ such that $\norm{e} \le \bar\gamma\norm{u}$ holds for all $(u,e)\in R_{ue}$.
\end{rem}

Theorem~\ref{thm:main} applies to a general semi-inner product space, which need not include a notion of time. Therefore, the notions of causality, boundedness, stability, and well-posedness do not come into play.

%%%%%%%%%%%%%%%%%%%%%%%%%%%%%%%%%%%%%%%%%%%%%%%%%%%%%%%%%%%%%%%%%%%%%%%%%%%%%%%%
%3. SPECIALIZATIONS OF MAIN RESULT (changing H)

%
\section{Specializing the main result}\label{sec:Specialization}
In this section, we specialize Theorem~\ref{thm:main} to recover a variety of classical results. We restrict our attention to discrete-time results in the interest of space, though continuous-time extensions are straightforward.
	
Recall the extended space $\ltwoe$, which is the real vector space of semi-infinite sequences $\Z_+\to \R^m$. Also recall the square-summable subset $\ltwo \subset \ltwoe$. Specifically,
\begin{align*}
\ltwoe &\defeq \set{(x[0],x[1],\dots)}{\vphantom{\bl(} x[k] \in \R^n\text{ for }k=0,1,\dots}, \\
\ltwo &\defeq \set{ x \in \ltwoe }{ \norm{x} \defeq \bbbl(\sum_{k=0}^\infty \normm{x[k]}^2\bbbr)^{1/2} < \infty }.
\end{align*}
Here, the indices $[k]$ play the role of \textit{time}. We now recall some standard definitions. 
For any $x\in\ltwoe$, we define the truncated signal $x_T \in \ltwo$ as follows.
	\[
	x_T[k]\defeq  \left\{
	\begin{array}{lr}
	x[k] & 0\le k \le T\\
	0 & k \ge T+1
	\end{array}
	\right.
	\]
An operator $G$ is said to be \textit{\textbf{causal}} if for any $T\ge 0$ and $f\in\ltwoe$, we have  $(Gf)_T=(Gf_T)_T$. We will now apply Theorem~\ref{thm:main} to the $\ltwoe$ space equipped
with a particular semi-inner product defined as 

\begin{rem}\label{rem:l2vsl2e}
It is  not fruitful to specialize Theorem~\ref{thm:main} to $\mathcal{X}=\ltwo$ because  \eqref{G} and \eqref{Phi} would imply $y_1, y_2 \in \ltwo$; we would be assuming the very thing we are trying to prove. We will instead specialize Theorem~\ref{thm:main} to $\mathcal{X} = \ltwoe$. 
\end{rem}

\paragraph{Causality and well-posedness.}
A possible concern in assessing stability of an interconnected systems in the form of Figure \ref{fig:StableNegativefeedback} is existence and uniqueness of solutions $e$ and $y$ for all choices of $u$. One solution is to use relations \cite{Zames_inputoutput1966_part1,vidyasagar2002nonlinear,Schaft2017L2,khong2018converse,safonov1980stability,teel1996graphs}, which avoids the issue entirely since all maps are invertible when viewed as relations. This amounts to using $\mathcal{C}=\mathcal{R}(\mathcal{X})$.  

Alternatively, one can assume that both $G$ and $\Phi$ are causal operators \cite{zhou1996robust,megretski1997system,Zames_inputoutput1966_part2,vidyasagar2002nonlinear,desoer2009feedback} rather than relations, which implies that if the closed-loop map exists, it must be causal as well~\cite[Prop.~1.2.14]{Schaft2017L2}. To work with causal operators, it is generally required to assume a notion of \textit{well-posedness}. In the $\ltwoe$ case, well-posedness requires the map $u \mapsto (e,y)$ to exist and be unique. Viewed through the lens of Theorem~\ref{thm:main}, if we let $\mathcal{C}$ be the set of causal operators, then our assumption of \textit{invariance} precisely corresponds to assuming well-posedness. Meanwhile, our assumption of \textit{completeness} is a technical condition that is automatically satisfied in $\ltwoe$.

\begin{defn}[cumulative semi-inner product]
Define $\ip{\cdot}{\cdot}_T$ to be the sum of the component-wise
inner products up to time $T$. That is, $\ip{x}{y}_T\defeq\ip{x_T}{y_T}$.
Also define the associated semi-norm $\norm{x}_T^2 \defeq \ip{x}{x}_T$.
\end{defn}

We now state a specialization of Theorem~\ref{thm:main} to the extended space $\ltwoe$, which we prove in Appendix~\ref{sec:appendix2}.

\begin{cor}[$\ltwo$ stability]\label{cor:cumulative}
	Let $M=M^\tp\! \in \R^{2\times 2}$ with $M\npreceq 0$ be given\footnote{The condition $M\npreceq 0$ is only required to prove \ref{Cit_ii}$\implies$\ref{Cit_i}. The case $M\preceq 0$ also corresponds to~\eqref{phi_cum} being degenerate.} and suppose $G:\ltwoe\to\ltwoe$ is a
	causal linear operator. The following statements are equivalent:
	\begin{enumerate}[label=(\roman*),itemindent=0pt,labelindent=0pt]
		\item\label{Cit_i} There exists $N=N^\tp \in \R^{2\times 2}$ satisfying $M + N \prec 0$ such that for all $\xi\in\ltwoe$ and $T \ge 0$, $G$ satisfies
		\begin{equation}\label{G_cum}
		\ip{ \bmat{G\xi\\ \xi} }{ N \bmat{G\xi\\\xi} }_T \ge 0.
		\end{equation}
		\item\label{Cit_ii} There exists $\gamma>0$ such that for all causal $\Phi:\ltwoe\to\ltwoe$ where the interconnection of $G$ and $\Phi$ is well-posed, if the following statement holds for all $T\ge 0$:
		\begin{equation}\label{phi_cum}
		\ip{ \bmat{\xi \\\Phi \xi} }{ M \bmat{\xi\\\Phi \xi} }_T \ge 0 \qquad\text{for all
		}\xi\in\ltwoe,
		\end{equation}
		then for all $(u,y) \in R_{uy}$ with $u\in\ell_2$,
			\begin{equation}\label{norm_l2}
		\norm{y} \le \gamma \norm{u}.
		\end{equation}
	\end{enumerate}
\end{cor}

\subsection{Recovering necessary and sufficient results}

Corollary~\ref{cor:cumulative} may now be applied to a variety of different scenarios by appropriately choosing $M,N\in\R^{2\times 2}$.

\begin{rem}[sign convention]\label{rem:Ntilde}
	Although we used the positive feedback sign convention in Figure~\ref{fig:StableNegativefeedback}, using the negative feedback convention instead simply amounts to replacing $N$ by $\tilde N$ in Theorem~\ref{thm:main} and Corollary~\ref{cor:cumulative}, where
	\[
	N \defeq \bmat{N_{11} & N_{12}\\ N_{21} & N_{22}}
	\quad\text{and}\quad
	\tilde N \defeq \bmat{N_{11} & -N_{12}\\ -N_{21} & N_{22}}.
	\]\vspace{0pt}
\end{rem}

\noindent Consider the classical passivity result by Vidyasagar (which is a sufficient-only result), and may be found in~\cite[Thm. 6.7.3.43]{vidyasagar2002nonlinear}. 
	\begin{thm}[Vidyasagar]\label{thm:vidyasagar}
		Consider the system 
		\[
		\left\{
		\begin{array}{lr}
		e_1  = u_1 - y_2, & y_1 = G e_1\\
		e_2 = u_2 + y_1, & y_2 = \Phi e_2
		\end{array}
		\right.
		\]
		Suppose there exist constants $\epsilon_1$, $\epsilon_2$, $\delta_1$, $\delta_2$ such that for all $\xi\in\ltwoe$ and for all $T\ge 0$
		\begin{subequations}\label{eq:passivityeq}
		\begin{align}
			\langle \xi, G\xi \rangle_T &\ge \epsilon_1 \| \xi\|_T^2 + \delta_1\| G\xi \|_T^2\\
			\langle \xi, \Phi \xi \rangle_T &\ge \epsilon_2 \| \xi\|_T^2 + \delta_2\| \Phi \xi \|_T^2
		\end{align}
		\end{subequations}
	Then the system is $\ltwo$-stable if $\delta_1+ \epsilon_2>0$ and $\delta_2+\epsilon_1>0$.		
	\end{thm}

To obtain a corresponding necessary and sufficient result using Corollary~\ref{cor:cumulative}, compare \eqref{G_cum}--\eqref{phi_cum} to~\eqref{eq:passivityeq}, which yields the following values of $\tilde{N}$, $N$, and $M$ (refer to Remark \ref{rem:Ntilde}). 
\[
\tilde{N} = \addtolength{\arraycolsep}{-2pt}\bmat{-\delta_1 & \frac{1}{2}\\\frac{1}{2} & -\epsilon_1}\!,\,
N =\bmat{-\delta_1 & -\frac{1}{2}\\ -\frac{1}{2} & -\epsilon_1}\!,\,
M = \bmat{-\epsilon_2& \frac{1}{2}\\ \frac{1}{2} & -\delta_2 }.
\]
Applying Corollary \ref{cor:cumulative}, we require $M+N\prec 0$; thus $\delta_1+ \epsilon_2>0$ and $\delta_2+\epsilon_1>0$, which recovers Theorem~\ref{thm:vidyasagar}.
		
Similar specializations of Corollary~\ref{cor:cumulative} apply to the small-gain theorem \cite[Thm.~5.6]{khalil2002nonlinear}, extended conic sector theorem~\cite{bridgeman2016extended}, circle criterion~\cite{jonsson2001lecture}, and other versions of passivity such as Vidyasagar~\cite[Thm.~6.6.58]{vidyasagar2002nonlinear} and Khong~\& Van~der~Schaft~\cite{khong2018converse}. See Table~\ref{Tab:comparison} for a summary of these results. The conditions of Corollary \ref{cor:cumulative} such as~\eqref{G_cum} can be checked via semidefinite programming~\cite{willems1971least}.

\begin{rem}
	Many results in the literature assume one of the systems is memoryless. As we will discuss in Section~\ref{sec:Existing}, this makes Corollary~\ref{cor:cumulative} sufficient-only. Nevertheless, $M$ and $N$ are the same in both cases.
\end{rem}

\subsection{Recovering sufficient-only results}\label{sec:Existing}

Here we discuss sufficient-only results from the literature that can also be obtained from Corollary~\ref{cor:cumulative} via a suitable relaxation. We discuss two such relaxations.

\paragraph{Memoryless systems.} If we restrict $\Phi:\ltwoe\to\ltwoe$ to be time-invariant and memoryless, it is equivalent to an operator $\phi:\R^m\to\R^m$ that operates pointwise in time. Consequently, if $\Phi$ satisfies a sector bound of the form
\begin{equation}\label{phi_pointwise}
\ip{\bmat{\xi\\\phi(\xi)}}{M\bmat{\xi\\\phi(\xi)}} \ge 0\quad\text{ for all }\xi\in \R^m,
\end{equation}
then $\Phi$ also satisfies the cumulative relationship~\eqref{phi_cum} for all $T$. Therefore, if we define condition \textit{(iii)} to be the same as condition \textit{\ref{Cit_ii}} from Corollary~\ref{cor:cumulative}, except~\eqref{phi_cum} is replaced by~\eqref{phi_pointwise}, then we have
\textit{(i)}$\iff$\textit{(ii)}$\implies$\textit{(iii)}.
So in general, if \textit{(i)} fails to hold, there must exist some $\Phi$ satisfying~\eqref{phi_cum} such that~\eqref{norm_l2} fails, but such a $\Phi$ need not be time-invariant or memoryless. Examples of this case in the classical literature include
\cite{narendra1973frequency,brockett1966status}. 

\paragraph{Nested sector bounds.} Another possible relaxation of Corollary~\ref{cor:cumulative} is to consider \textit{nested sectors} for one of the systems. For example, define \textit{(i')} to be the same as \textit{(i)} except $N$ is replaced by some $\hat N \preceq N$. Similarly, define \textit{(ii')} to be the same as \textit{(ii)} except $M$ is replaced by some $\hat M \preceq M$. Then, we have the implications: \textit{(i')}$\implies$\textit{(i)}$\iff$\textit{(ii)}$\implies$\textit{(ii')}.
The implication \textit{(i')}$\implies$\textit{(ii')} cannot be reversed in general, and is therefore a sufficient-only condition.

\begin{table*}[t]
\caption{
Existing sufficient conditions for stability that can be transformed into necessary-and-sufficient conditions via Corollary~\ref{cor:cumulative}. We use a positive feedback convention; for negative feedback, replace $N$ by $\tilde N$ as in Remark~\ref{rem:Ntilde}.}
\label{Tab:comparison}
\centering
\small
\begin{tabular}{lccc}
\toprule
\leftbox{\textbf{Name of Theorem}} &
\midbox{$\boldsymbol{M}$} &
\midbox{$\boldsymbol{N}$} &
\midboxx{$\boldsymbol{M+N\prec 0}$} \\
\midrule
\leftbox{\textbf{Conic sector theorem}\\
\cite[Thm.~2a,~all three cases]{Zames_inputoutput1966_part1} or\\
\cite[Thm.~3.1,~all parts of Case~1]{bridgeman2016extended}}&
\matbox{\frac{-(a+\Delta)(b-\Delta)}{b-a-2\Delta} & \frac{-a-b}{2(b-a-2\Delta)}\\\frac{-a-b}{2(b-a-2\Delta)} & \frac{-1}{b-a-2\Delta} }&
\matbox{\frac{ab}{b-a+2ab\delta} & \frac{a+b}{2(b-a+2ab\delta)}\\ \frac{a+b}{2(b-a+2ab\delta)} & \frac{(1+a\delta)(1-b\delta)}{b-a+2ab\delta} } &
\midbox{$a<b$, and either $\delta=0,\Delta>0$ or $\delta>0,\Delta=0$.}\\[4mm]
\leftbox{\textbf{Extended conic sector thm.}\\ \cite[Thm. 3.1, all parts of Case 2]{bridgeman2016extended}} & 
\matbox{\frac{(a-\Delta)(b+\Delta)}{b-a+2\Delta} & \frac{a+b}{2(b-a+2\Delta)}\\\frac{a+b}{2(b-a+2\Delta)} & \frac{1}{b-a+2\Delta} }&
\matbox{\frac{-ab}{b-a-2ab\delta} & \frac{-a-b}{2(b-a-2ab\delta)}\\ \frac{-a-b}{2(b-a-2ab\delta)} & \frac{-(1-a\delta)(1+b\delta)}{b-a-2ab\delta} } &
\midbox{Same as above}\\[4mm]
\leftbox{\textbf{Extended passivity}\\%
\cite[Thm.~6.6.58]{vidyasagar2002nonlinear}} &
\matbox{-\epsilon_2 & \frac{1}{2} \\ \frac{1}{2} & -\delta_2} &
\matbox{-\delta_1 & -\frac{1}{2}\\ -\frac{1}{2} &-\epsilon_1} &
\midbox{$\delta_1+ \epsilon_2 >0$ and $\delta_2 + \epsilon_1 >0$} \\[4mm]
\leftbox{\textbf{Small gain theorem}\\ \cite[Thm.~5.6]{khalil2002nonlinear}} &
\matbox{\gamma_2 & 0\\0 &-1/\gamma_2 } &
\matbox{-1/\gamma_1 & 0 \\ 0 & \gamma_1 } &
\midbox{$\gamma_1 \gamma_2 < 1$} \\[2mm]
\bottomrule
\end{tabular}
\end{table*}

\section{Weighted stability result}\label{sec:Exponential}
In this section, we present a specialization of Theorem~\ref{thm:main} that leads to a new necessary and sufficient condition for \textit{weighted stability}, which in turn is sufficient for exponential stability.
For a fixed $\rho\in(0,1]$, define the set $\ltwo^\rho \subset \ltwo$ of sequences $\{x[k]\}$ such that $\sum_{k=0}^\infty \rho^{-2k}\normm{x[k]}^2<\infty$. This can be thought of as enforcing that $x$ converge to zero exponentially fast. Define the corresponding semi-inner products as
$
\ip{x}{y}_{\rho,T} \defeq \sum^T_{k=0} \rho^{-2k}\ip{x[k]}{y[k]}
$.
Analogously to how we derived Corollary~\ref{cor:cumulative}, we have:

\begin{cor}[Weighted stability]\label{cor:exponential}
	Let $M =M^\tp \in \R^{2\times 2}$ with $M\npreceq 0$ and $\rho \in (0,1]$ be given. Suppose $G:\ltwoe\to\ltwoe$ is causal and linear. The following are equivalent.
\begin{enumerate}[label=(\roman*)]
	\item There exists $N=N^\tp \in \R^{2\times 2}$ satisfying $M + N \prec 0$ such
	that for all $\xi\in\ltwoe^\rho$ and $T  \ge 0$, $G$ satisfies
	\begin{equation}\label{G_rho}
	\ip{ \bmat{G\xi\\ \xi} }{ N \bmat{G\xi\\\xi} }_{\rho,T} \ge 0.
	\end{equation}
	\item There exists $\gamma >0$ such that for all causal $\Phi:\ltwoe^\rho\to\ltwoe^\rho$ where the interconnection of $G$ and $\Phi$ is well-posed, if the following condition holds for all $T\ge 0$
	\begin{equation}\label{phi_rho}
	\ip{ \bmat{\xi \\\Phi \xi} }{ M \bmat{\xi\\\Phi \xi} }_{\rho,T} \ge 0 \qquad\text{for all
	}\xi \in \ltwoe^\rho,
	\end{equation}
	then for all $(u,y) \in R_{uy}$ with $u\in\ltwo^\rho$, we have
	\begin{equation}\label{norm_expt2}
		\norm{y}_{\rho} \le \gamma \norm{u}_\rho.
	\end{equation}
\end{enumerate}
\end{cor}

The weighted stability guarantee~\eqref{norm_expt2} in Corollary~\ref{cor:exponential} states that when inputs to the system tend to zero exponentially quickly (in the sense that $\lim_{k\to\infty} \rho^{-k} u_k = 0$ for some $\rho \in (0,1]$), then so do the outputs $y$. Under additional assumptions about $G$, this condition implies exponential stability, as detailed in Proposition~\ref{prop:ross}.

\begin{prop}[\!\!{\cite[Prop. 5]{boczar2017exponential}}]\label{prop:ross}
Suppose $G$ is a discrete-time LTI system and has a minimal realization with state $x[k]$. If the interconnection in Figure~\ref{fig:StableNegativefeedback} has weighted stability with weight $\rho \in (0,1]$, then there exists some $c>0$ such that for any initial $x[0]$ and with $u=0$, we have
\[
\normm{ x[k] } \le c \rho^k \normm{ x[0] } \qquad \text{for }k=0,1,\dots 
\]
\end{prop}

The converse of Proposition~\ref{prop:ross}, that exponential stability implies weighted stability, does not hold in general. For example, if $G=0$ then we have exponential stability for any $\Phi$. Proving such a converse result typically requires stronger assumptions on $\Phi$ such as Lipschitz continuity~\cite[\S6.46]{vidyasagar2002nonlinear}.

\section{Conclusion}
In this paper, we introduced a robust stability theorem (Theorem~\ref{thm:main}) framed in a general semi-inner product space. Our result unifies many existing results, including passivity, small-gain, and circle theorems. This includes both necessary-and-sufficient as well as sufficient-only versions, and relation-based as well as operator-based notions of systems. Our theorem also leads to a new result on weighted stability (Corollary~\ref{cor:exponential}).

%%%%%%%%%%%%%%%%%%%%%%%%%%%%%%%%%%%%%%%%%%%%%%%%%%%%%%%%%%%%%

\section{Acknowledgments}
Authors would like to thank R.~Boczar, \mbox{L.~Bridgeman}, A.~Packard, A.~Rantzer, P.~Seiler, A.~van~der~Schaft, B.~Van~Scoy, S.~Z.~Khong and M.~Vidyasagar for helpful discussions and comments.

\appendix
%%%%%%%%%%%%%%%%%%%%%%%%%%%%%%%%%%%%%%%%%%%%%%%%%%%%%%%%%%%%%%%%%%%%%%%%%%%%%%%%%
\section{Proof of Theorem~\ref{thm:main}}\label{sec:appendix1}
%2. PROOF OF MAIN THEOREM

\paragraph{Sufficiency:} \ref{thm_it_i}$\implies$\ref{thm_it_ii}. In other words, we show that \eqref{G} and \eqref{Phi} imply \eqref{norm}.
This part of the proof resembles \cite[Thm. 1]{mccourt2010control}. Pick any $(u,y) \in R_{uy}$ and let $e$ be an associated signal such that $(u,y,e)$ satisfies~\eqref{interconnect}.
Let $\xi=e_2$ in~\eqref{Phi} and let $\xi=e_1$ in~\eqref{G}. Using \eqref{interconnect} to eliminate $e_1$, $e_2$, Equations~\eqref{Phi} and \eqref{G} become:
\begin{align*}
\ip{ \bmat{u_2+y_1\\y_2} }{ M \bmat{u_2+y_1\\y_2} } &\ge 0\quad\text{and}\\
\ip{ \bmat{y_1\\ u_1+y_2} }{ N \bmat{y_1\\u_1+y_2} } &\ge 0.
\end{align*}
Sum the inequalities above and collect terms to obtain
\begin{multline*}
\ip{ \bmat{y_1\\y_2}\! }{ (M \!+\! N)\! \bmat{y_1 \\ y_2} } +2
\ip{ \bmat{y_1\\y_2}\! }{ \addtolength{\arraycolsep}{-2pt}\bmat{N_{12}&M_{11}\\N_{22}&M_{21}}\!\bmat{u_1\\u_2} } \\
+ \ip{ \bmat{u_1\\u_2}\! }{ \addtolength{\arraycolsep}{-2pt} \bmat{N_{22}&0\\0&M_{11}}\!\bmat{u_1\\u_2} }
\ge 0.
\end{multline*}
Since $M+N \prec 0$ by assumption, There exists $\eta > 0$ such that $M + N
\preceq -\eta I$. Applying this inequality together with Cauchy--Schwarz, we
obtain:
\begin{equation}\label{gamineq}
-\eta\norm{y}^2 + 2r \norm{y} \norm{u} + q \norm{u}^2 \ge 0,
\end{equation}
where 
$
r \defeq \left\| \left( \begin{smallmatrix} N_{12}&M_{11}\\N_{22}&M_{21}\end{smallmatrix} \right) \right\|_2
$
and
$
q \defeq \left\| \left( \begin{smallmatrix} N_{22}&0\\0&M_{11} \end{smallmatrix} \right) \right\|_2
$
are standard spectral norms in $\R^{2\times 2}$. Dividing \eqref{gamineq} by $\norm{u}^2$ and viewing the left-hand side as a polynomial in $\norm{y}/\norm{u}$, we observe that it always has two real roots and the inequality in~\eqref{gamineq} implies the bound
\[
\norm{y} \le \frac{1}{\eta}\left( r + \sqrt{r^2 +
\eta q} \right) \norm{u}.
\]
\paragraph{Necessity:} \ref{thm_it_ii}$\implies$\ref{thm_it_i}. This part of the proof is based on the S-lemma, which relates sets of points defined by quadratic inequalities. See~\cite{jonsson2001lecture,megretski_treil} and references therein.
We use a generalization of the S-lemma to semi-inner product spaces based on a Hilbert space version due to Hestenes~\cite[Thm.~7.1,~p.~354]{hestenes} and similar to \cite{khong2018converse}. This relies on the following notion of quadratic form.

\begin{defn}
	Let $\mathcal{V}$ be a real vector space. A \emph{quadratic form} $Q$ is a
	function $Q:\mathcal{V}\to \R$ that has associated with it a function
	$\tilde Q:\mathcal{V}\times \mathcal{V}\to\R$ such that the following properties hold
	for all $x,y,z\in \mathcal{V}$ and $a,b\in\R$.
	\begin{enumerate}[label=(P\arabic*),itemindent=1ex]
		\item $Q(x) = \tilde Q(x,x)$
		\item $\tilde Q(x,y) = \tilde Q(y,x)$
		\item $\tilde Q(x,ay+bz) = a\tilde Q(x,y)+b\tilde Q(x,z)$
		\item $Q(ax+by) = a^2Q(x) + 2ab \tilde Q(x,y) + b^2 Q(y)$
	\end{enumerate}
\end{defn}

\begin{lem} \label{sproc3}
	Let $\mathcal{V}$ be a real vector space and let $S\subseteq \mathcal{V}$ be a
	subspace.
	Let $\sigma_0$ and $\sigma_1$ be quadratic forms. The following statements are
	equivalent.
	\begin{enumerate}[label=(S\arabic*),itemindent=1ex]
		\item For all $x \in S$, we have $\sigma_1(x) \ge 0 \!\implies\! \sigma_0(x)
		\le 0$. \label{itb1}
		\item There exists $\tau\ge 0$ such that for all $x\in S$, we have\\
		$\hphantom{1}\sigma_0(x) + \tau \sigma_1(x) \le 0$. \label{itb2}
	\end{enumerate}
\end{lem}

\begin{proof}
	Omitted. A similar result on Hilbert spaces is proved in~\cite[Thm.~7.1,~p.~354]{hestenes} and extends immediately to semi-inner product spaces. In our version, both inequalities in \ref{itb1} are non-strict, which means the standard regularity conditions are no longer required.
\end{proof}

\noindent We can now prove our result. We will use $\Theta$ to denote a generic tuple $(u,y,e) \in (\mathcal{X}^2)^3$. Define the sets:
\begin{align*}
S'_\Phi &\defeq \set{ \Theta \in (\mathcal{X}^2)^3 }{ \text{Equations \eqref{1a}--\eqref{1d} hold} }, \\
S &\defeq \set{ \Theta \in (\mathcal{X}^2)^3 }{ \text{Equations \eqref{1a}, \eqref{1c}, \eqref{1d} hold} }.
\end{align*}
Note that $S'_\Phi$ depends on $\Phi$ but $S$ does not.
Since $G$ is linear by assumption, it follows that $S$ is a subspace. Moreover, 
$
\bigcup_{\Phi\in\mathcal{C}} S'_{\Phi} = S
$.
To see why, first observe that $S'_\Phi \subseteq S$ for all $\Phi$ by definition. To prove the opposite inclusion, given any $\Theta\in S$, there exists $\Phi\in\mathcal{C}$ such that $y_2 = \Phi e_2$, which is possible because $\mathcal{C}$ is complete.
Define the quadratic forms on $S\to\R$:
\begin{align*}
\sigma_0(\Theta) &\defeq  \norm{y}^2 - \gamma^2\norm{u}^2, &
\sigma_1(\Theta) &\defeq \ip{ \bmat{e_2\\y_2} }{ M \bmat{e_2\\y_2} }.
\end{align*}
Item \ref{thm_it_ii} from Theorem~\ref{thm:main} states that for all $\Phi \in \mathcal{C}$ which satisfy $\sigma_1(\Theta) \ge 0$ for all $\Theta \in S'_\Phi$ , we have  $\sigma_0(\Theta) \le 0$.
But since $\bigcup_{\Phi\in\mathcal{C}} S'_{\Phi} = S$, Item \ref{thm_it_ii} is equivalent to the statement that for all $\Theta\in S$, we have $\sigma_1(\Theta) \ge 0 \implies \sigma_0(\Theta) \le 0$. Since $S$ is a subspace, we may apply Lemma~\ref{sproc3} and conclude that there exists $\tau \ge 0$ such that
\begin{equation}\label{eq:losslessS}
\text{for all }\Theta\in S,\qquad \sigma_0(\Theta) + \tau \sigma_1(\Theta) \le 0.
\end{equation}
Substituting the definitions of $\sigma_0$ and $\sigma_1$ into \eqref{eq:losslessS} yields
\begin{equation}\label{eq:lossSy}
\norm{y}^2 - \gamma^2\norm{u}^2
+ \tau \ip{ \bmat{e_2\\y_2} }{ M \bmat{e_2\\y_2} } \le 0.
\end{equation}
Let $\bar{S} \subseteq S$ be the subspace of $S$ with $u=0$.  Restricting~\eqref{eq:lossSy} to $\bar{S}$, we have
$y_2 = e_1$ and $e_2 = y_1 = Ge_1$. Making these substitutions results in:
\[
\norm{Ge_1}^2 + \norm{e_1}^2
+ \tau\ip{ \bmat{Ge_1\\e_1} }{ M \bmat{Ge_1\\e_1} } \le 0, 
\]
for all $e_1\in\mathcal{X}$. If $\tau=0$, then from nontriviality of $\mathcal{X}$, the above inequality clearly cannot hold for all $e_1$, a
contradiction. Therefore it must be the case that $\tau > 0$. Rearranging and
dividing by $\tau$, we obtain:
\[
\ip{ \bmat{Ge_1\\e_1} }{ (-\tfrac{1}{\tau}I - M ) \bmat{Ge_1\\e_1} } \ge 0
\qquad\text{for all }e_1\in\mathcal{X}.
\]
Define $N \defeq -\tfrac{1}{\tau}I - M$. Then we have $M+N =
-\tfrac{1}{\tau}I \prec 0$, which is \eqref{G} and so we have proven~\ref{thm_it_i} of Theorem \ref{thm:main}.
\qedhere

\section{Proof of Corollary~\ref{cor:cumulative}}\label{sec:appendix2}
	To prove \ref{Cit_i}$\implies$\ref{Cit_ii}, suppose~\eqref{G_cum} and~\eqref{phi_cum} hold. For each $T\ge 0$, apply Theorem~\ref{thm:main} with $\mathcal{X}=\ltwoe$ and the semi-inner product $\ip{\cdot}{\cdot}_T$ and choose $\mathcal{C}$ to be the set of causal operators. From the proof of Theorem~\ref{thm:main}, $\gamma$ only depends on the choice of $M$ and $N$, and not on the choice of semi-inner product. Well-posedness (feedback invariance of $\mathcal{C}$) also does not depend on the semi-inner product, so fixing $\Phi$ and $u\in\ltwo$ always yields the same $R_{uy}$ regardless of $T$. Therefore, we have $\norm{y}_T \le \gamma\norm{u}_T$ for all $T\ge 0$ and $\gamma,y$ are independent of $T$. Since $u\in\ltwo$, letting $T\to\infty$ implies $y\in\ltwo$ and we obtain \eqref{norm_l2}.
	
	To prove \ref{Cit_ii}$\implies$\ref{Cit_i}, suppose~\eqref{phi_cum} and~\eqref{norm_l2} hold. Since $\mathcal{C}$ is the set of causal operators, the map $u\to y$ is causal, which implies that $\norm{y}_T \le \gamma\norm{u}_T$ holds for all $T\ge 0$ (see, e.g.~\cite[Lem.~6.2.11]{vidyasagar2002nonlinear} for a proof). Apply Theorem~\ref{thm:main} as before to obtain~\eqref{G_cum}. It remains to show that the same $N$ can be used for all $T \ge 0$. We provide an outline. From the proof of Theorem~\ref{thm:main}, we obtain a $\tau_T$ from which $N_T = -\tfrac{1}{\tau_T}I-M$. It is straightforward to show that any $\tau \ge \tau_T$ may be used in the place of $\tau_T$. One can also show that if $M\npreceq 0$, the sequence $\{\tau_T\}_{T\ge 0}$ is uniformly bounded so we can pick $N = \sup_{T\ge 0}\bigl(-\tfrac{1}{\tau_T}I-M\bigr)$.\qedhere

%%%%%%%%%%%%%%%%%%%%%%%%%%%%%%%%%%%%%%%%%%%%%%%%%%%%%%%%%%%%%%%%%%%%%%%%%%%%%%%%
\bibliographystyle{abbrv}
\begin{small}
	\bibliography{refs}
\end{small}

\end{document}